\documentclass[leqno,11pt]{scrartcl}
\usepackage[utf8]{inputenc}

\usepackage{amsmath,amssymb,amsthm}
\usepackage{amscd}
\usepackage{setspace} %%% Einstellung des Zeilenvorschubs mit \spacing{ } %%%
\usepackage{enumerate}
\usepackage{array}
\usepackage{tabularx}
\usepackage{multirow}
\usepackage{booktabs}

\theoremstyle{definition}

\newtheorem{remark}{Remark}

\theoremstyle{plain}
\newtheorem{theorem}{Theorem}

\newtheorem{corollary}{Corollary}

\theoremstyle{remark}

\newcommand{\C}{\mathbb{C}}

\newcommand{\K}{\mathbb{K}}
\newcommand{\N}{\mathbb{N}}

\newcommand{\Q}{\mathbb{Q}}
\newcommand{\R}{\mathbb{R}}
\newcommand{\Z}{\mathbb{Z}}
\newcommand{\Acal}{\mathcal{A}}

\newcommand{\Dcal}{\mathcal{D}}

\newcommand{\Hcal}{\mathcal{H}}
\newcommand{\Ical}{\mathcal{I}}

\newcommand{\Ncal}{\mathcal{N}}

\newcommand{\diag}{\operatorname{diag}\,}

\newcommand{\oh}{{\scriptstyle{{\cal O}}}}

\let\geq\geqslant

\begin{document}

\begin{center}
\begin{huge}
\begin{spacing}{1.0}
\textbf{The Hilbert Modular Group and Orthogonal Groups}  
\end{spacing}
\end{huge}

\bigskip
by
\bigskip

\begin{large}
\textbf{Adrian Hauffe-Waschbüsch\footnote{Adrian Hauffe-Waschbüsch, Lehrstuhl A für Mathematik, RWTH Aachen University, D-52056 Aachen, adrian.hauffe@rwth-aachen.de}} and
\textbf{Aloys Krieg\footnote{Aloys Krieg, Lehrstuhl A für Mathematik, RWTH Aachen University, D-52056 Aachen, krieg@rwth-aachen.de}}
\end{large}
\vspace{0.5cm}\\
June 2022
\vspace{1cm}
\end{center}
\begin{abstract}
\textbf{Abstract.}We derive an explicit isomorphism between the Hilbert modular group and certain congruence subgroups on the one hand and particular subgroups of the special orthogonal group $SO(2,2)$ on the other hand. The proof is based on an application of linear algebra adapted to number theoretical needs.
\end{abstract}
\noindent\textbf{Classification: 11F41, 11F55}
\vspace{2ex}\\

\newpage
\section{Introduction}

A generalization of the classical elliptic moduar group $SL_2(\Z)$ was introduced by Blumenthal \cite{Bl} more than 100 years ago. $SL_2(\oh_\K)$ is the ordinary Hilbert modular group over a totally real number field  $\K$. Meanwhile the theory of Hilbert modular forms has been developped into various directions (cf. \cite{Fr}, \cite{Gar}, \cite{vG}). On the other hand Borcherds \cite{Bo2} established a product expansion for modular forms on $SO(2,n)$. If $n=2$ and $\K$ is a real-quadratic field these groups are basically isomorphic. 
More precisely $PSL_2(\oh_\K)$ is isomorphic to discriminant kernel of the orthogonal group $SO(2,2)$ (cf. Theorem \ref{theorem_2}). Borcherds theory has led to important examples of Hilbert modular forms (cf. \cite{Bru}), which in some cases allow to describe the graded ring of Hilbert modular forms (cf. \cite{May}, \cite{Will2}). 
As Borcherds products are usually related with discriminant kernels, it is interesting to describe the associated congruence subgroups of $PSL_2(\R)$ precisely. 

In this paper we derive an explicit isomophism based on linear algebra. Our approach can be extended to congruence subgroups. The results below can serve as an explicit dictionary, when passing between Borcherds theory and classical Hilbert modular forms.

Given a non-degenerate symmetric even matrix $T\in\Z^{m\times m}$ let 
\[
 SO(T;\R):= \{U\in SL_m(\R);\; U^{tr} TU=T\}
\]
denote the attached special orthogonal group. Let $SO_0(T;\R)$ stand for the connected component of the identity matrix $I$ and $SO_0(T;\Z)$ for the subgroup of integral matrices. The \emph{discriminant kernel} 
\begin{gather*}\tag{1}\label{gl_1}
 \Dcal(T;\Z):= \{U\in SO_0(T;\Z);\; U\in I + \Z^{m\times m} T\}
\end{gather*}
is clearly a normal subgroup of $SO_0(T;\Z)$. Given $N\in\N$ we set 
\begin{gather*}\tag{2}\label{gl_2}
 T_N:= \begin{pmatrix}
        0 & 0 & N \\ 0 & T & 0 \\ N & 0 & 0
       \end{pmatrix}
\end{gather*}
for the orthogonal sum with the rescaled hyperbolic plane.

\section{The normalizer of the Hilbert modular group}
\label{sect_2}

Throughout this paper let $\K=\Q(\sqrt{m})$, $m\in\N$, $m>1$ squarefree, be a real-quadratic number field with \emph{ring of integers} and \emph{discriminant}
\begin{gather*}\tag{3}\label{gl_3}
 \oh_\K = \Z+\Z\omega_\K,\quad \omega_\K = \begin{cases}
                                            (m+\sqrt{m})/2, \\
                                            m+\sqrt{m},
                                           \end{cases} \quad d_\K = \begin{cases}
								    m, & \text{if} \; m\equiv 1\bmod{4}, \\
								    4m, & \text{else}.
								    \end{cases}
\end{gather*}
The non-trivial automorphism of $\K$ is given by 
\begin{gather*}\tag{4}\label{gl_4}
 \K\to \K, \quad a=\alpha +\beta\sqrt{m} \; \mapsto \; a' = \alpha -\beta\sqrt{m}.
\end{gather*}
If $\ell \in\N$, $\sqrt{\ell\,}\notin \K$ we extend \eqref{gl_4} to $\K(\sqrt{\ell\,})$ via $\sqrt{\ell\,}' = \sqrt{\ell\,}$. Moreover $a\gg 0$ means that $a\in \K$ is \emph{totally positive}, i.e. $a>0$ and $a'>0$.

\vspace{1ex}
The (ordinary) \emph{Hilbert modular group} (cf. \cite{Bl}, \cite{Fr}) is given by
\[
  \Gamma_\K : = SL_2(\oh_\K).
\]
At first we are going to describe the \emph{normalizer}
\[
 \Ncal_\K : = \{M\in SL_2(\R);\; M^{-1} \Gamma_\K M = \Gamma_\K\}
\]
of $\Gamma_\K$ in $SL_2(\R)$. Therefore let $\Ical(L)$ denote the ideal generated by the entries of a matrix $0\neq L \in \K^{2\times 2}$.

\begin{theorem}\label{theorem_1} %% Theorem 1
 If $\K$ is a real-quadratic number field, the normalizer $\Ncal_\K$ is equal to 
\begin{gather*}\tag{5}\label{gl_5}
 \left\{\frac{1}{\sqrt{\det L}}L;\; L\in\oh^{2\times 2}_\K,\, \det L>0,\, \Ical(L)^2 = \oh_\K \det L\right\}.
\end{gather*}
\end{theorem}
\begin{proof}
 Clearly the matrices in \eqref{gl_5} belong to $\Ncal_\K$. As $\Gamma_\K$ contains a $\Z$-basis of $\oh^{2\times 2}_\K$, we conclude that any $M\in \Ncal_\K$ satisfies
 \[
  M^{-1} \oh^{2\times 2}_\K M = \oh^{2\times 2}_\K.
 \]
Inserting the classical basis elements, we see that the product of any two entries of $M$ belongs to $\oh_K$. If $\alpha$ is any non-zero entry of $M$, we get
 \[
  M = \frac{1}{\sqrt{\alpha^2}}L,\; L\in \oh^{2\times 2}_\K, \quad \det L = \alpha^2, \quad \Ical(L)^2 = \oh_\K\alpha^2.
 \]
Thus $M$ belongs to \eqref{gl_5}.
\end{proof}

Now we define the group
\[
 \Sigma_\K = \left\{\frac{1}{\sqrt{\ell}} L;\; \ell \in \N,\, L\in \oh^{2\times 2}_\K\right\}\cap SL_2(\R)\supseteq \Gamma_\K.
\]
Note that $M_1, M_2 \in\Sigma_\K$ satisfy
\[
 (M_1 M_2)' =  \pm M'_1 M'_2 
\]                                         
due to the above extension of \eqref{gl_4}.

Let $\Hcal$ denote the upper half-plane in $\C$. Then a matrix $M=\left(\begin{smallmatrix}
                                                                         \alpha & \beta \\ \gamma & \delta
                                                                        \end{smallmatrix}\right) 
\in \Sigma_\K$ acts on $\Hcal^2$ via 
\begin{gather*}\tag{6}\label{gl_6neu}
 \binom{\tau_1}{\tau_2} \mapsto \binom{M\langle\tau_1\rangle}{M'\langle\tau_2\rangle}, \quad M \langle\tau\rangle = \frac{\alpha \tau + \beta}{\gamma\tau + \delta}.
\end{gather*}
It is well-known that a subgroup $\Gamma \subseteq \Sigma_\K$  acts dicontinuously if and only if the embedded group 
\begin{gather*}\tag{7}\label{gl_7neu}
\bigl\{\pm(M,M');\;M\in\Gamma\bigr\} \subseteq \bigl(SL_2(\R) \times SL_2(\R)\bigr)\big /\{\pm(I,I)\} 
\end{gather*}
is discrete (cf. \cite{Fr}, I.2.1). 
The maximal discontinuous extension $\Gamma^*_\K$ of $\Gamma_\K$ is called the \emph{Hurwitz-Maaß extension}. It is described in \cite{Maa}, \cite{Ba}, \cite{Ha} as
\begin{gather*}\tag{8}\label{gl_6}
  \begin{split}
   \Gamma^*_\K & = \left\{\frac{1}{\sqrt{\det L}}L;\; L\in\oh^{2\times 2}_\K,\, \det L\gg 0,\, \Ical(L)^2 = \oh_\K \det L\right\} \\
   & = \left\{\frac{1}{\sqrt{\ell}} L;\; \ell \in \N,\, L\in \oh^{2 \times 2}_\K,\, \det L = \ell,\, \Ical(L)^2 = \oh_\K \ell\right\} \\
   & = \bigcup_{\ell\in\N\;\text{squarefree}, \,\ell|d_\K} V_\ell \Gamma_\K \subseteq\K,
   \end{split}
\end{gather*}
where we define the generalized Atkin-Lehner matrices by
\begin{gather*}\tag{9}\label{gl_7}
 V_\ell = \frac{1}{\sqrt{\ell\,}} \begin{pmatrix}
                                 \nu\ell & \mu(m+\sqrt{m}) \\ m-\sqrt{m} & \ell
                                \end{pmatrix}, \;\; \nu,\mu\in\Z,\quad \nu\ell-\mu m(m-1)/\ell = 1.
\end{gather*}
If $\Acal_\ell = \Z\ell + \Z\omega_\K$ denotes the integral ideal in $\oh_\K$ of reduced norm $\ell$ for a sqaurefree divisor $\ell$ of $d_\K$, we have 
\begin{gather*}\tag{10}\label{gl_10}
 V_\ell \Gamma_\K = \Gamma_\K V_\ell = \tfrac{1}{\sqrt{\ell\,}} \Acal^{2\times 2}_\ell \cap SL_2(\R).
\end{gather*}
Note that for squarefree divisors $k,\ell$ of $d_\K$, we have 
\begin{align*}
   & \tfrac{1}{\sqrt{k}} \Acal_k = \oh_\K \; \Leftrightarrow \; k = 1 \;\; \text{or}\;\; k = m, \\ %\tag{9}\label{gl_9} \\
   & \tfrac{1}{\sqrt{k}} \Acal_k \cdot \tfrac{1}{\sqrt{\ell\,}} \Acal_\ell = \tfrac{1}{\sqrt{f}} \Acal_f, \;\; f= \frac{k\ell}{\gcd(k,\ell)^2}.%\tag{10}\label{gl_10}
 \end{align*}
Note that due to \eqref{gl_6} and Theorem \ref{theorem_1}, in some cases there exists a matrix $M_0\in \Ncal_\K$, $M_0\notin \Gamma^*_\K$. If the fundamental unit $\varepsilon_0$ satisfies $\varepsilon_0\varepsilon'_0 = -1$, we may choose
\begin{gather*}\tag{11}\label{gl_11}
 M_0 = \frac{1}{\sqrt{\varepsilon_0}} \begin{pmatrix}
                                       \varepsilon_0 & 0 \\ 0 & 1
                                      \end{pmatrix}.
\end{gather*}
If $m= \alpha^2+\beta^2$ for some $\alpha,\beta \in\N$, $\alpha$ odd, we choose $\nu,\mu\in\Z$ satisfying $\nu\alpha-2\mu \beta = 1$, $u: = \beta+\sqrt{m} \in \oh_\K$, $uu' = -\alpha^2$,
\begin{align*}
 & L = \begin{pmatrix}
        \mu\alpha+\nu u & \mu u \\ u & \alpha
       \end{pmatrix}, \quad \det L = u> 0, \\
 & \Ical(L) = \oh_\K \alpha + \oh_\K u, \quad \Ical(L)^2 = \oh_\K u, \;\; u'< 0.
\end{align*}
Thus we may choose
\begin{gather*}\tag{12}\label{gl_12}
 M_0 = \tfrac{1}{\sqrt{u}} L
\end{gather*}
in this case. If we apply the theory of ambiguous ideals, then \cite{Le}, 7.6 and 7.8, imply

\begin{corollary}\label{corollary_1} %%% Corollary 1
 Let $\K$ be a real-quadratic number field with fundamental unit $\varepsilon_0$. If $\varepsilon_0 \varepsilon'_0 = 1$ and $p\mid d_\K$ for some prime $p\equiv 3\bmod{4}$, we have
 \[
  \Ncal_\K = \Gamma^*_\K.
 \]
In any other case
\[
 \Ncal_\K = \Gamma^*_\K \cup M_0 \Gamma^*_\K
\]
holds with $M_0$ from \eqref{gl_11} and \eqref{gl_12}.
\end{corollary}

We can quote Maaß \cite{Maa} or apply the results from \eqref{gl_6} and \eqref{gl_10} in order to obtain
\[
 [\Gamma^*_\K:\Gamma_\K] = 2^{\nu-1}, \quad \nu=\sharp\{p\;\text{prime}; p\mid d_\K\}.
\]
\begin{remark}\label{remark_1} %%% Remark 1
a) The field $\widehat{\K}$ generated by the entries of the matrices in $\Gamma^*_\K$ is given by
\[
 \widehat{\K} = \Q(\sqrt{p};\; p\;\text{prime}, \; p\mid d_\K)
\]
according to \eqref{gl_6} and \eqref{gl_7}. Hence an analog of Theorem 4 in \cite {KRaW} also holds in this case:
\[
 \widehat{\K}\supseteq \Q \;\text{\emph{is unramified outside}}\; 2d_\K.
\]
b) If $\varepsilon \gg 0$ is a unit in $\oh_\K$, then \eqref{gl_6} yields
\[
 \frac{1}{\sqrt{2+\varepsilon + \varepsilon'}} \begin{pmatrix}
                                                \varepsilon +1 & 0 \\ 0 & \varepsilon'+1
                                               \end{pmatrix}
 \in \Gamma^*_\K.
\]
Thus the squarefree kernel $q$ of $2+\varepsilon + \varepsilon'$ always divides $d_\K$.
\end{remark}

\section{The Hilbert modular group as an orthogonal group}

If $T\in\Z^{2\times 2}$ is an even symmetric matrix  with $\det T< 0$ the associated half-space $\Hcal_T$ is defined to be
\[
 \Hcal_T: = \{z= x+iy\in\C^2;\; y^{tr} Ty> 0,\, y_1>0\}.
\]
Given $\widetilde{M}\in SO_0(T_N;\R)$ (cf. \eqref{gl_2}) we will always assume the form
\begin{gather*}\tag{13}\label{gl_13}
\widetilde{M} = \begin{pmatrix}
                        \widetilde{\alpha} & a^{tr} T & \widetilde{\beta} \\ b & K & c \\ \widetilde{\gamma}  & d^{tr}T & \widetilde{\delta}
                       \end{pmatrix}, \;\; \widetilde{\alpha},\widetilde{\beta},\widetilde{\gamma},\widetilde{\delta}\in\R.
\end{gather*}
It is well-known (cf. \cite{G4}, \cite{K1}) that $\widetilde{M}$ acts on $\Hcal_T$ via
\begin{gather*}\tag{14}\label{gl_14}
\begin{split}
 z\mapsto \widetilde{M}\langle z\rangle & := \bigl(\widetilde{M}\{z\}\bigr)^{-1}\bigl(-\tfrac{1}{2} z^{tr} Tz\cdot b + Kz + c\bigr),  \\
 \widetilde{M}\{z\} & := -\tfrac{1}{2}\widetilde{\gamma} z^{tr} Tz + d^{tr} Tz + \widetilde{\delta}.
 \end{split}
\end{gather*}
Note that $\Hcal^2$, where $\Hcal$ denotes the upper half-plane in $\C$, is the orthogonal half-space $\Hcal_P$, $P=\left(\begin{smallmatrix}
                                                                      0 & 1 \\ 1 & 0
                                                                     \end{smallmatrix}\right)$. 
Considering \cite{K1}, sect. 5, we obtain a surjective homomorphism of the groups
\begin{align*}
 \Omega: SL_2(\R) \times SL_2(\R) & \to SO_0(P_1;\R), \\
 (U,V) & \mapsto \begin{pmatrix}
                  \alpha FVF & \beta FV \\ \gamma VF & \delta V
                 \end{pmatrix}, \quad
 U = \begin{pmatrix}
      \alpha & \beta \\ \gamma & \delta
     \end{pmatrix},
 F = \begin{pmatrix}
      -1& 0 \\ 0 & 1
     \end{pmatrix},    
\end{align*}
with kernel $\{\pm(I,I)\}$ satisfying
\begin{gather*}\tag{15}\label{gl_15}
 \Omega(U,V)\left\langle\binom{\tau_1}{\tau_2}\right\rangle = \binom{U\langle\tau_1\rangle}{V\langle\tau_2\rangle}.
\end{gather*}

Thus we have for $M\in \Sigma_\K$
\begin{gather*}\tag{16}\label{gl_16}
 \Omega(M,M') = \begin{pmatrix}
       \alpha \alpha' & (-\alpha' \beta  & -\alpha \beta')P & -\beta \beta' \\
       -\alpha \gamma' & \alpha \delta' & \beta \gamma' & \beta \delta' \\
       -\alpha' \gamma & \beta' \gamma & \alpha' \delta & \beta' \delta \\
       -\gamma \gamma' & (\gamma' \delta & \gamma \delta')P & \delta \delta'
      \end{pmatrix}
\in SO_0(P_1;\R).
\end{gather*}
Note that $\Omega(M,M') \neq -I$ for all these $M$. 
Now choose a basis $B=(u,v)$ of $\K$ over $\Q$ and consider the base change
\[
 \binom{\tau_1}{\tau_2} = Gz, \quad G= \begin{pmatrix}
                                                       u & v \\ u' & v'
                                                      \end{pmatrix}, \quad
 \widehat{G} = \begin{pmatrix}
                1 & 0 & 0 \\ 0 & G & 0 \\ 0 & 0 & 1
               \end{pmatrix}.
\]
We set
\begin{gather*}\tag{17}\label{gl_17}
\begin{split}
  \widetilde{M} = \widehat{G}^{-1} \Omega(M,M') \widehat{G} & \in SO_0(S_1;\R), \quad 
 S = G^{tr} PG = \begin{pmatrix}
      2uu' & uv' + u'v \\ uv' + u'v & 2vv'
     \end{pmatrix},  \\
  & \widetilde{M}\{z\} = (\gamma\tau_1 + \delta) (\gamma'\tau_2 + \delta').
\end{split}
\end{gather*}
If we define
\[
 \varphi_B : \K\to \Q^2,\;\; \alpha u + \beta v \mapsto \binom{\alpha}{\beta},
\]
a straightforward calculation yields that the description of $\widetilde{M}$ in \eqref{gl_13} is given by
\begin{gather*}\tag{18}\label{gl_18}
 \begin{cases}
  & \widetilde{\alpha} = \alpha\alpha', \;\widetilde{\beta} = -\beta\beta',\; \widetilde{\gamma} = -\gamma\gamma',\; \widetilde{\delta} = \delta\delta',  \\
  & a = \varphi_B(-\alpha'\beta),\;b = \varphi_B(-\alpha\gamma'),\;c = \varphi_B(\beta\delta'),\; d = \varphi_B(\gamma'\delta),  \\
  & K \;\text{is the matrix that represents the endomorphism } \\
  & f_M: \K\to \K, \; w\mapsto \alpha \delta' w + \beta\gamma' w', \\
  & \text{with respect to the basis}\; B. \\
 \end{cases}
\end{gather*}
Recalling the definition of the discriminant kernel from \eqref{gl_1}, we obtain 
\begin{theorem}\label{theorem_2} %% Theorem 2
 Let $\K$ be a real-quadratic field and let $B=(u,v)$ be a $\Z$-basis of $\oh_\K$ with Gram matrix $S$ from \eqref{gl_17}. Then the mappings
 \begin{align*}
  \widetilde{\phi}_B : \Sigma_\K/\{\pm I\} & \to SO_0(S_1;\Q)/\{\pm I\}, \; \pm M \mapsto \pm \widetilde{M}, \\
  \widetilde{\phi}_B : \Gamma^*_\K/\{\pm I\} & \to SO_0(S_1;\Z)/\{\pm I\}, \; \pm M\mapsto \pm \widetilde{M},
 \end{align*}
where $\widetilde{M}$ is defined by \eqref{gl_18}, are isomorphisms of the groups. Moreover the mapping 
\[
 \phi_B: \Gamma_\K/\{\pm I\} \to \Dcal(S_1;\Z),\; \pm M\mapsto \widetilde{M},
\]
is an isomorphism of the groups.
\end{theorem}
\begin{proof}
 We get $\widetilde{\phi}_B\bigl(\Sigma_\K/\{\pm I\}\bigr) \subseteq SO_0(S_1;\Q)/\{\pm I\}$ and  $\widetilde{\phi}_B\bigl(\Gamma^*_\K/\{\pm I\}\bigr) \subseteq SO_0(S_1;\Z)/\{\pm I\}$ directly from \eqref{gl_18}. As $\widetilde{\phi}_B$ maps the group actions in \eqref{gl_14} and \eqref{gl_15} onto each other, $\widetilde{\phi}_B$ is an injective homomorphism of the groups. 
 In the notation of \cite{K1} we easily see that $SO_0(S_1;\Q)$ is generated by the matrices
 \begin{align*}
  & T_\lambda, \widetilde{T}_\lambda, \, \lambda \in\Q^2, \; \diag (\ell,1,1,1/\ell),\, \ell\in \N,   \\
  &  R_K, \, K = \begin{pmatrix}
               a & bm \\ b & a
              \end{pmatrix},
  \; w = a+b\sqrt{m} \in\K,\, w>0,\, ww' = 1.
 \end{align*}
 They appear as images of the matrices 
 \[
  \pm\begin{pmatrix}
      1 & \lambda \\ 0 & 1
     \end{pmatrix}, \;
  \pm\begin{pmatrix}
      1 & 0 \\ \lambda & 1
     \end{pmatrix}, \, \lambda \in \K,\;   
  \pm\frac{1}{\sqrt{\ell}}\begin{pmatrix}
      \ell & 0 \\ 0 & 1
     \end{pmatrix}, \, \ell \in \N,\;
   \frac{1}{\sqrt{2+2a}}\begin{pmatrix}
      w+1 & 0 \\ 0 & w'+1
     \end{pmatrix}   
 \]
 in $\Sigma_\K$. Hence $\widetilde{\phi}_B$ is an isomorphism for the rationals. 
 As $\widetilde{\phi}_B^{-1}\bigl(SO_0(S_1;\Z)/\{\pm I\}\bigr)$ is a discontinuous subgroup of $PSL_2(\R)$ containing $\Gamma^*_\K/\{\pm I\}$, we obtain equality from the result of Maaß \cite{Maa}. 
 
 In view of $-I\not\in \phi_B\bigl(\Gamma_\K/\{\pm I\}\bigr)$ we conclude that $\phi_B$ is an injective homomophism of the groups with $\phi_B\bigl(\Gamma_\K/\{\pm I\}\bigr)\subseteq SO_0(S_1;\Z)$. We get  
 \begin{gather*}\tag{19}\label{gl_19}
  (K-I) \in \Z^{2\times 2} S
 \end{gather*}
 if and only if $F_M = f_M -id$ satisfies
 \[
  \left(\phi_B (F_M(u)), \phi_B(F_M(v))\right) \in \Z^{2\times 2} S.
 \]
 If $X\in \Z^{2\times 2}$ satisfies 
 \[
  \bigl(\sqrt{d_\K} u, \sqrt{d_\K} v\bigr) = (u,v)X
 \]
 the latter condition becomes equivalent to 
 \[
  \bigl(\phi_B(F_M(\sqrt{d_\K}u)), \phi_B(F_M(\sqrt{d_\K}v))\bigr) = \bigl(\phi_B(F_M(u)), \phi_B(F_M(v))\bigr)X \subseteq \Z^{2\times 2} SX = d_\K \Z^{2\times 2}
 \]
 as 
 \[
  SX = G^{tr} PGX = G^{tr}\sqrt{d_\K} JG = \pm d_\K J, \; J=\begin{pmatrix}
                                                             0 & -1 \\ 1 & 0
                                                            \end{pmatrix}.
 \]
 Thus \eqref{gl_19} holds if and only if
\[
   F_M\bigl(\sqrt{d_\K} \oh_\K\bigr) \subseteq d_\K \oh_\K.
\]
 This is easily demonstrated for $M\in\Gamma_\K$ using $\alpha\delta -\beta \gamma = 1$. 
 If $\ell$ is a squarefree divisor of $d_\K$, $\ell\neq 1,m$ we verify 
 \[
  F_{V_\ell}\bigl(\sqrt{d_\K}\bigr) \notin d_\K \oh_\K.
 \]
 Thus $\phi_B$ becomes an isomorphism, too.
\end{proof}

Now we apply the results to congruence subgroups.

\begin{corollary}\label{corollary_2} %%% Corollary 2
 Let $N\in\N$ and $B=(u,v)$ be a $\Z$-basis of $\oh_\K$ with Gram matrix $S$ from \eqref{gl_17}.\\[1ex]
 a) $\phi_B $ maps
 \[
  \left\{ M=\begin{pmatrix}
             \alpha & \beta \\ \gamma & \delta
            \end{pmatrix} 
  \in \Gamma_\K;\; \alpha\alpha'\equiv \delta\delta' \equiv 1 \bmod{N},\,\gamma \in N\oh_\K\right\}\Big/\{\pm I\}
 \]
onto
\[
 F_N \Dcal(S_N;\Z) F_N^{-1}, \quad F_N = \diag(1,1,1,N).
\]
b) $\phi_B$ maps the principal congruence subgroup
\[
\{M\in \Gamma_\K;\; M\equiv \varepsilon I\bmod{N\oh_\K},\, \varepsilon \in \Z,\, \varepsilon^2\equiv 1\bmod{N}\}\big/\{\pm I\}
\]
onto
\[
 \Dcal(NS_1;\Z).
\]
\end{corollary}
\begin{proof}
 a) Apply Theorem \ref{theorem_2} and \eqref{gl_18}.  \\[1ex]
 b) Given $M$ in the principal congruence subgroup, then $\phi_B(\pm M) \in \Dcal(NS_1;\Z)$ is a consequence of \eqref{gl_18} and Theorem \ref{theorem_2}. On the other hand $\pm M\in \phi_B^{-1}\left(\Dcal(NS_1;\Z)\right)$ implies $\beta,\gamma\in N\oh_\K$ due to \eqref{gl_18}. Considering the representing matrix of $w\mapsto \alpha \delta' w$ and using $\alpha \delta \equiv 1 \bmod{N\oh_\K}$ we obtain
 \[
  M\equiv \varepsilon I\bmod{N\oh_\K},\; \varepsilon \in \Z,\quad \varepsilon^2\equiv 1 \bmod{N}.
 \]
 \vspace{-8ex}\\
\end{proof}

We add a Remark

\begin{remark}\label{Remark_2} %%% Remark 2
 $\oh_\K$ with the symmetric bilinear form $(\alpha, \beta)\mapsto \alpha \beta' + \alpha'\beta$ is a maximal even lattice. Thus Theorem \ref{theorem_2} corresponds to the results on $SO(2,n)$, $n\geq 3$ in \cite{KSch}.
\end{remark}

\section{The general Hilbert modular group}

Given an integral ideal $0\neq \Ical \subseteq \oh_\K$ the \emph{(general) Hilbert modular group with respect to} $\Ical$ is defined by
\[
  \Gamma_\K(\Ical):=\left\{\begin{pmatrix}
                           \alpha & \beta \\ \gamma & \delta
                          \end{pmatrix}
 \in SL_2(\K);\;\alpha,\delta\in\oh_\K,\, \beta\in \Ical,\, \gamma\in\Ical^{-1}\right\}.
\]
 Given a basis $B$ of $\K$ then Theorem \ref{theorem_2} and \eqref{gl_18} immediately imply
\[
  \phi_B\bigl(\Gamma_\K(n\Ical)/\{\pm I\}\bigr) = H_n \phi_B \bigl(\Gamma_\K(\Ical)/\{\pm I\}\bigr) H^{-1}_n,
\]
whenever $n\in\N$ and $H_n = \phi_B\left(\pm \frac{1}{\sqrt{n}}\left(\begin{smallmatrix}
       n & 0 \\ 0 & 1
      \end{smallmatrix}\right)\right)=\diag(n,1,1,1/n)$. Hence we may assume that $\Ical$ is a \emph{primitive} ideal, i.e.
\[
 \tfrac{1}{n} \Ical \subseteq \oh_\K, \; n\in\N \;\Rightarrow\; n=1.
\]
In this case $\Ical$ contains a $\Z$-basis
 \begin{gather*}\tag{20}\label{gl_20}
  N,t+\omega_\K,\; N\in \N,\; t\in\Z, \quad N\mid(t+\omega_\K)(t+\omega'_\K),
 \end{gather*}
where $N$ is the reduced norm of $\Ical$ (cf. \cite{Le}, Proposition 4.2). It follows from \cite{Ha}, \cite{He} or \cite{Maa} that the so-called \emph{Maaß-Hurwitz extension} $\Gamma^*_\K(\Ical)$ is a maximal discontinuous subgroup of $SL_2(\R)$ containing $\Gamma_\K (\Ical )$ with index 
\[
 2^{\nu-1},\;\; \nu= \sharp\{p\;\text{prime};\; p\mid d_\K\}.
\]
It is given by
\begin{align*}
 \Gamma^*_\K (\Ical) & = \biggl\{\frac{1}{\sqrt{\det L}} L;\;L\in\begin{pmatrix}
                                                                 \oh_\K & \Ical \\ \Ical^{-1} & \oh_\K
                                                                \end{pmatrix}, \,
 \det L\gg 0,\, \Ical(L)^2 = \oh_\K\det L\biggr\} \\[0.5ex]
 & = \bigcup_{\ell\in\N\,\text{squarefree},\,\ell|d_\K} V_{\ell} \Gamma_\K(\Ical).
\end{align*}
One may choose
\[
 V_\ell = \frac{1}{\sqrt{\ell}} \begin{pmatrix}
                                 \nu\ell & \mu u N \\ u' & \ell
                                \end{pmatrix}, \;\;
 u=\frac{\ell}{\gcd(\ell,N)} + m+\sqrt{m},\; \nu,\mu \in\Z,\; \nu\ell-\mu N uu'/\ell = 1.
\]
Just as in sect. \ref{sect_2} we have
\[
 V_\ell \Gamma_\K(\Ical) = \Gamma_\K(\Ical) V_\ell = \frac{1}{\sqrt{\ell}} \begin{pmatrix}
                                                                                   \Acal_\ell& \Acal_\ell \Ical \\ \Acal_\ell \Ical^{-1} & \Acal_\ell
                                                                                  \end{pmatrix}
 \cap SL_2 (\R) 
\]
as well as 
\[
 V_\ell \Gamma_\K(\Ical) = \Gamma_\K (\Ical) \;\Longleftrightarrow \; \ell=1\;\,\text{or}\;\, \ell = m.
\]
\begin{theorem}\label{theorem_3} %% Theorem 3
 Let $\K$ be a real-quadratic field and let $\Ical\subseteq \oh_\K$ be a primitive ideal with reduced norm $N$. Assume that $B = (u,v)$ is a $\Z$-basis of $\Ical$ with Gram-matrix $S$ from \eqref{gl_17} and let $T = \frac{1}{N} S$. Then the mappings
 \begin{align*}
  \psi_B: \Gamma_\K(\Ical)\big /\{\pm I\} & \to \Dcal(T_1;\Z), \; \pm M\mapsto H^{-1}_N \widetilde{M} H_N,  \\
  \widetilde{\psi}_B: \Gamma^*_\K(\Ical)\big /\{\pm I\} & \to SO_0(T_1;\Z)\big /\{\pm I\}, \; \pm M\mapsto \pm H^{-1}_N \widetilde{M} H_N,
 \end{align*}
 where $H_N=\diag(N,1,1,1)$ and $\widetilde{M}$ is given by \eqref{gl_18}, are isomorphisms of the groups.
\end{theorem}
\begin{proof}
 Proceed exactly as in the proof of Theorem \ref{theorem_2}. Observe that $\Ical^{-1} = \frac{1}{N} \Ical'$ and verify that
 \[
  (K-I) \in\Z^{2\times 2} T \;\Longleftrightarrow \; F_M\bigl(\sqrt{d_\K} \Ical\bigr) \subseteq d_\K \Ical.
 \]
 Use the basis \eqref{gl_20} of $\Ical$ in order to verify that 
 \[
  F_{V_\ell} \bigl(\sqrt{d_\K} \Ical\bigr) \subseteq d_\K \Ical \;\Longleftrightarrow \; \ell = 1 \;\text{or}\; \ell = m,
 \]
 whenever $\ell$ is a squarefree divisor of $d_\K$.
\end{proof}

We obtain an immediate application to congruence subgroups of $\Gamma_\K$. 

\begin{corollary}\label{corollary_3} %% Corollary 3
Let $N$ be a squarefree divisor of $d_\K$. Then the principal congruence subgroup
\[
\left\{M\in\Gamma_\K;\; M\equiv \varepsilon I\bmod{\Acal_N},\, \varepsilon\in \Z,\, \varepsilon^2\equiv 1\bmod{N}\right\}\big/\{\pm I\}
\]
is isomorphic to 
\[
 G^{-1}_N \Dcal\bigl(T_N;\Z)G_N , \;\, G_N=\diag(1,N,N,1), \;\, T = 
 \begin{pmatrix}
  2N & \omega_\K + \omega'_\K \\ \omega_\K + \omega'_\K & 2\omega_\K \omega'_\K/N
 \end{pmatrix}
\]
via $\phi_B$, where $B = (N,\omega_\K)$.
\end{corollary}
\begin{proof}
Apply Theorem \ref{theorem_3} to $\Ical = \Acal_N = \Z N + \Z\omega_\K$ and follow the proof of Corollary \ref{corollary_2}\,a). 
\end{proof}

We add a final Remark.

\begin{remark}\label{Remark_3} %%% Remark 3
The congruence conditions with respect to $N$ as opposed to more general $\Ical$ in Corollaries \ref{corollary_2} and \ref{corollary_3} are  equivalent to the 
 restriction to ideals $\Ical$ satisfying $\Ical = \Ical'$. It is not only motivated by technical reasons, but also by the fact that symmetric Hilbert modular forms, i.e. $f(\tau_2,\tau_1) =f(\tau_1,\tau_2)$, can be defined for these subgroups.
\end{remark}

\noindent\textbf{Statements and Declarations.} The authors declare that there are no competing conflicts of interests and that they did not receive any funding.
The authors did not use any data from a data repository

\nocite{Hu}

\vspace{6ex}

%============================================

\bibliography{bibliography_krieg_2021} 
\bibliographystyle{plain}

\end{document}